\documentclass[reqno]{amsart}
\usepackage{a4wide,calrsfs}
\usepackage{amsmath,bbm}
\usepackage{mathtools}
\usepackage{color}
\usepackage{epsfig}
\usepackage{subfigure}
\usepackage{amssymb}
\usepackage{enumerate}
\usepackage{graphicx}
\usepackage{subfigure}
\usepackage{rotating}
\usepackage{stmaryrd}
\usepackage{upgreek}
\usepackage{amsthm}
\usepackage{amsopn}
\usepackage{epstopdf}
\usepackage{cite}
\DeclareMathAlphabet{\mathcal}{OMS}{cmsy}{m}{n}

\newtheorem{definition}{Definition}[section]
\newtheorem{theorem}{Theorem}[section]

\newtheorem{remark}{Remark}[section]
\newtheorem{lemma}{Lemma}[section]

\newcommand{\bvarphi}{\boldsymbol{\varphi}}
\newcommand{\R}{\mathbb{R}}

\newcommand{\bu}{\mathbf{u}}
\newcommand{\bw}{\mathbf{w}}
\newcommand{\bF}{\mathbf{F}}
\newcommand{\bG}{\mathbf{G}}
\newcommand{\bH}{\mathbf{H}}
\newcommand{\bL}{\mathbf{L}}
\newcommand{\bV}{\mathbf{V}}
\newcommand{\bW}{\mathbf{W}}
\newcommand{\bx}{\mathbf{x}}

\newcommand{\bm}{\mathbf{m}}
\newcommand{\bn}{\mathbf{n}}
\newcommand{\bq}{\mathbf{q}}
\newcommand{\bv}{\mathbf{v}}

\newcommand{\diver}{\mathrm{div}\,}
\newcommand{\curl}{\mathrm{curl}\,}
\DeclareMathOperator*{\esssup}{ess\,sup}

\allowdisplaybreaks
\numberwithin{equation}{section}

\title[]
{A  priori estimates for the system modelling nonhomogeneous asymmetric fluids}

\author[A. Coronel]{An{\'\i}bal\ Coronel$^\dag$}
\author[E. Fern\'andez-Cara]{Enrique Fern\'andez-Cara$^\ddag$}
\author[M. Rojas-Medar]{Marko Rojas-Medar$^\S$}
\author[A. Tello]{Alex Tello$^\dag$}

\thanks{$^\dag$
GMA, Departamento de Ciencias B\'asicas, Facultad de Ciencias,
Universidad del B\'{\i}o-B\'{\i}o, Campus Fernando May, Chill\'{a}n, Chile.}
\email{acoronel@ubiobio.cl,alextello21@gmail.com}

\thanks{$^\ddag$
Dpto. EDAN, University of Sevilla, 
Aptdo. 1160, 41080 Sevilla, Spain}
\email{cara@us.es}

\thanks{$^\ddag$
Instituto de Alta Investigaci\'on, 
Universidad de Tarapac\'a, Casilla 7D, Arica, Chile}
\email{marko.medar@gmail.com}

\date{\today}
\keywords{A priori estimates; 
nonhomogeneous asymmetric fluids; 
micropolar fluids;  
navier-stokes system}

\begin{document}

\begin{abstract}
In this paper, we prove some a priori estimates for 
a system of partial differential equations arising in the nonstationary
flow of a nonhomogeneous incompressible asymmetric fluid
in a bounded domain with smooth boundary. 
The unknowns of the system are the velocity
field of the fluid particles,
the angular velocity of rotation of the fluid particles,
the mass density of the fluid and the pressure distribution. 
For the density functions we consider 
the application of the Helmholtz decomposition.

\end{abstract}
\maketitle

\numberwithin{equation}{section}
\newtheorem{thm}{Theorem}[section]
\newtheorem{lem}[thm]{Lemma}
\newtheorem{prop}[thm]{Proposition}
\newtheorem{cor}[thm]{Corollary}
\newtheorem{defn}{Definition}[section]
\newtheorem{conj}{Conjecture}[section]
\newtheorem{exam}{Example}[section]
\newtheorem{rem}{Remark}[section]
\allowdisplaybreaks

\section{Introduction}

\subsection{Scope}
The well known micropolar fluids or also called asymmetric fluids 
are a widely class of fluids which are relevant in 
many industrial applications and in 
several areas of science, see 
for instance \cite{antontsev_1990,boldrini_1996,boldrini_2003,braz_2014,braz_2007,
conca_2002,lukaszewicz_1990,lukaszewicz_1989,rojasmedar_2005,vitoriano_2013}.
Consequently, there is several mathematical models to describe the phenomenon.
In particular, a wide class of that models are based on the assumptions of Navier-Stokes
type with a non-symmetric Cauchy tensor, 
see \cite{lukaszewicz_1990} for details. In that case the model is a system
of differential equations for the linear momentum, the the angular momentum, the 
continuity equation and the incompressibility condition.
More precisely, let us consider a nonhomogeneous viscous incompressible asymmetric fluid 
on a bounded and regular domain $\Omega\subset\mathbb{R}^3$, 
with boundary~$\partial\Omega$  and outward unit normal vector $\bn$. 
Then, the motion of the fluid in a finite time $t\in [0,T]$,
is described by the velocity field $\bu$, the angular velocity of rotation of 
the fluid particles $\bw$, the mass density 
$\rho$ and the pressure distribution $p$,  satisfying the system
\begin{eqnarray}
&&(\rho \bu)_t+\diver(\rho \bu\otimes \bu)
	-(\mu +\mu_r)\Delta \bu+\nabla p
	=2\mu_r \curl \bw+\rho \bF,
	\quad\mbox{in}\quad Q_T:=\Omega\times [0,T],
\label{eq:momento_lineal}
\\
&&\diver (\bu)=0,
	\quad\mbox{in}\quad Q_T,
\label{eq:incompresibilidad}
\\
&&\rho \Big(\bw_t+\diver(\bu\otimes \bw)\Big)-(c_a+c_d)\Delta \bw
	-(c_0+c_d-c_a)\nabla\diver \bw+
	4\mu_r \bw
\nonumber
\\
&&
\hspace{2cm}
=2\mu_r \curl \bu+\rho \bG,
	\quad\mbox{in}\quad Q_T,
\label{eq:momento_angular}
\\
&&
 \rho_t+\bu\cdot \nabla \rho=0,
	\quad\mbox{in}\quad Q_T,
\label{eq:ecuacion_continuidad}
\\
&&
 \bu(\bx,0)=\bu_0(\bx),\quad  \bw(\bx,0)=\bw_0(\bx),\quad  \rho(\bx,0)=\rho_0(\bx),
	\quad\mbox{on}\quad \Omega,
\label{eq:direct_problem_ic}
\\
&&
 \bu(\bx,t)=\bw(\bx,t)=0,
	\quad\mbox{on}\quad \Sigma_T:=\partial\Omega\times [0,T],
\label{eq:direct_problem_bc}
\end{eqnarray} 
where $\bF$ and $\bG$ are 
the density functions, modelling the vector
external sources for the linear and the angular momentum of particles; 
the constant $\mu$ is the usual Newtonian viscosity
and the constants $\mu_r,c_0$ and $c_d$ are 
the additional viscosities  satisfying 
\begin{align*}
\mu>0,\quad \mu_r>0, \quad c_a+c_d>0,\quad c_0+c_d>c_a,
\end{align*}
and related to the lack of symmetry of the stress tensor \cite{braz_2007,lukaszewicz_book}. 
The differential notation is the standard ones, i.e.  the
symbols $\bu_t,\bw_t$ and $\rho_t$ denote the time derivatives and
$\nabla,\Delta,\diver$ and $\curl$ denote the gradient, Laplacian,
divergence and rotational operators, respectively.

In this paper we want to obtain a priori estimates for the 
weak solution of the following mathematical model  for asymmetric fluids
when the external sources satisfy the specific decomposition
\begin{eqnarray}
\bF(\bx,t)=f(t)(\nabla h(\bx,t)- \bm(\bx,t)), 
\quad
\bG(\bx,t)=g(t)\bq(\bx,t),
\quad
\mbox{in $\Omega_T$,}
	\label{eq:helmholtz_f_uno}
\end{eqnarray}
where $\bm$ and $\bq$ are given functions and $f,g$ and $h$
are unknown functions such that
\begin{eqnarray}
&&\diver(\rho\nabla h)=\diver(\rho \bm), 
	\quad\mbox{in}\quad\Omega, 
	\label{eq:helmholtz_f_dos}
\\
&&\frac{\partial h}{\partial \bn}=\bm\cdot  \bn, 
	\quad\mbox{on}\quad \Sigma_T,
	\label{eq:helmholtz_f_tres}
\\
&&\int_{\Omega}h(\bx,t)d\bx=0.
	\quad t\in[0,T].
	\label{eq:helmholtz_f_cuatro}
\end{eqnarray}
We notice that this type of representation of $\bF$ is a consequence of Helmholtz
decomposition \cite{fan_2008}.

\subsection{Notation}
In order to define the weak solution we recall the standard notation of some functional
spaces and operators frequently used 
to study the Navier-Stokes system, see \cite{antontsev_1990,lions_book_1,lions_book_2,temam_1977}
for details. 
The Banach space of measurable functions that are $p$-integrable 
in the sense of Lebesgue or are essentially bounded on $\Omega$ are denoted by $L^p(\Omega)$
for $p\in[1,\infty[$ and by $L^{\infty}(\Omega)$, respectively. 
We recall that, the norms in $L^p(\Omega)$ for  $p\in [1,\infty[$ and
$p=\infty$ are defined as follows
\begin{eqnarray*}
\|u\|_{L^p(\Omega)}:=\left(\int_{\Omega}|u(\bx)|^pd\bx\right)^{1/p}
\mbox{ and }
\|u\|_{L^{\infty}(\Omega)}:=\esssup_{\bx\in\Omega}|u(\bx)|,
\end{eqnarray*}
respectively.
The notation $W^{m,q}(\Omega),$ where $m\in\mathbb{N}$ and $q\ge 1$ is used for the Sobolev
space consisting of all functions in $L^q(\Omega)$ having all distributional derivatives
of the first $m$ orders belongs to $L^q(\Omega)$, i.e.
\begin{eqnarray*}
W^{m,q}(\Omega):=\Big\{\;u\in  L^q(\Omega)\;:\; 
	D^{\alpha}u\in L^q(\Omega)\mbox{ for } |\alpha|=1,\ldots,m\;\Big\}.
\end{eqnarray*}
The norm of $W^{m,q}(\Omega)$ is naturally defined as follows
\begin{eqnarray*}
\|u\|_{W^{m,q}(\Omega)}
	:=\left(\sum_{k=0}^{m}\sum_{|\alpha|=k}
	\|D^{\alpha}u\|^q_{L^q(\Omega)}\right)^{1/q}
\mbox{ and }
\|u\|_{W^{m,\infty}(\Omega)}
	:=\max_{0\le |\alpha|\le m}\|D^{\alpha}u\|_{L^\infty(\Omega)}.
\end{eqnarray*}
The vector-valued spaces $[L^2(\Omega)]^3$ and
$[W^{m,p}(\Omega)]^3$ are defined as usual and by simplicity are
denoted by bold characters, i.e. $\bL^2(\Omega)$ and
$\bW^{m,p}(\Omega)$, respectively.
Also, we use the following rather common notation 
in mathematical theory of fluid mechanics:
\begin{eqnarray*}
&&H^m(\Omega)=W^{m,2}(\Omega), 
\quad
H^m_0(\Omega)=\overline{C_0^\infty(\Omega)}^{\|\cdot\|_{H^1(\Omega)}},
\\
&&
\mathcal{V}(\Omega)=\Big\{\bv\in (C_0^\infty(\Omega))^3\;:\;
\diver(\bv)= 0\mbox{ in }\Omega\;\Big\},
\\
&&
\bH=\overline{\mathcal{V}(\Omega)}^{\|\cdot\|_{\bL^2(\Omega)}}
\quad
\mbox{and}
\quad
\bV=\overline{\mathcal{V}(\Omega)}^{\|\cdot\|_{\bH_0^1(\Omega)}},
\end{eqnarray*}
where $\overline{A}^{\|\cdot\|_{B}}$ denotes the closure of $A$ in $B$.
Furthermore, for a given Banach space $X$, we denote by $L^r(0,T;X)$, $r\ge 1$,
the  Banach space of the $X$-valued functions having bounded the
norm $\|\cdot\|_{L^r(0,T;B)}$ defined as follows
\begin{eqnarray*}
\|u\|_{L^r(0,T;B)}:=\left(\int_0^T\|u(\cdot,t)\|_{B}^r dt\right)^{1/r}
\mbox{ and }
\|u\|_{L^{\infty}(0,T;B)}:=\esssup_{t\in[0,T]}\|u(\cdot,t)\|_{B}.
\end{eqnarray*}
Concerning to the linear operators, we define the operators: $A,L_0$
and $L$.
We denote by $A$ the stokes operator defined from 
$D(A):=\bV\cap \bH^2(\Omega)\subset \bH$ to $ \bH$
by $A\bv=P(-\Delta \bv)$, where $P$ is the orthogonal projection of 
$\bL^2(\Omega)$ onto
$\bH$ induced by the Helmholtz decomposition of $\bL^2(\Omega)$. It is well known that
$A$ is an unbounded linear and positive self-adjoint  operator, and is characterized by 
the following identity
\begin{eqnarray}
(A\bw,\bv)=(\nabla \bw,\nabla \bv),
	\quad \forall \bw\in D(A),\quad \bv\in V,
\label{eq:ident_stokes}
\end{eqnarray}
where $(\cdot,\cdot)$ is the usual scalar product in $\bL^2(\Omega)$.
In second place, we consider the strongly uniformly elliptic operators
$L_0$ and $L$ defined on $D(L_0)=D(L)=\bH^1_0(\Omega)\cap \bH^2(\Omega)$
as follows
\begin{eqnarray}
L_0\mathbf{z}=-(c_a+c_d)\Delta \mathbf{z}-(c_0+c_d-c_a)\nabla \diver \mathbf{z}
\quad\mbox{and}\quad
L\mathbf{z}=L_0z+4\mu_r \mathbf{z}.
\label{eq:def_of_L_and_L0}
\end{eqnarray}
Note that $L$ is a positive operator under the assumption $c_0+c_d>c_a$, see 
\eqref{eq:momento_angular}.

\subsection{Presentation of the main result}

\begin{definition}
\label{def:strong_solutions_dp} \cite{boldrini_2003}
Consider that the  functions $f,g,\bm$ and $\bq$ are given.
Then,
a collection of functions
$\{\bu,\bw,\rho,p,h\}$ is a 
solution  \eqref{eq:momento_lineal}-\eqref{eq:helmholtz_f_cuatro} if
there exists $T_*\in ]0,T]$ such that the functions satisfy
the following four conditions
\begin{enumerate}
\item[(a)] Regularity conditions:
\begin{eqnarray}
&&
\bu\in C^0\Big([0,T_*[;D(A)\Big)\cap C^1\Big([0,T_*[;\bH\Big),
\label{eq:reg_of_u}
\\
&&
\bw\in C^0\Big([0,T_*[;D(L)\Big)\cap C^1\Big([0,T_*[;\bL^2(\Omega)\Big)
\quad\mbox{and}
\label{eq:reg_of_w}
\\
&&
\rho \in C^1(\overline{\Omega}\times [0,T_*[).
\label{eq:reg_of_rho}
\end{eqnarray}
\item[(b)] Integral identities:
\begin{eqnarray}
&&\Big((\rho \bu)_t,\bv\Big)
	-\int_{\Omega}\rho \bu\otimes \bu:\nabla \bv\; d\bx
	+(\mu +\mu_r)\Big(A \bu,\bv \Big)
\nonumber\\
&&
\hspace{2cm}
=2\mu_r \Big(\curl \bw,\bv\Big)+
	\Big(\rho \bF,\bv\Big),\quad
	\mbox{for } t\in ]0,T_*[\mbox{ and }\forall \bv\in V,
\label{eq:formulacion_variacional_u}
\\
&&\Big((\rho \bw)_t,\bvarphi\Big)
	-\int_{\Omega}\rho \bu\otimes \bw:\nabla \bvarphi \; d\bx
	+\Big(L \bw,\bvarphi \Big)
\nonumber\\
&&
\hspace{2cm}
=2\mu_r \Big(\curl \bu,\bvarphi\Big)+
	\Big(\rho \bG,\bvarphi\Big),\quad
	\mbox{for } t\in ]0,T_*[\mbox{ and }\forall \bvarphi\in 
	\bH^1_0(\Omega).
\label{eq:formulacion_variacional_w}
  \end{eqnarray}
\item[(c)] Mass conservation: $\rho$ satisfies \eqref{eq:ecuacion_continuidad} for
$(x,t)\in \overline{\Omega}\times [0,T_*[$.
\item[(d)] Initial condition:
$\bu,\bw,\rho$ satisfies \eqref{eq:direct_problem_ic} for $\bx\in\Omega.$
\end{enumerate}

\end{definition}

The main result of the paper is the following theorem
\begin{theorem}
\label{teo:global_estimates}
Assuming that the following hypotheses 
\begin{enumerate}
\item[(H$_0$)] The initial  density $\rho_0$ 
is such 
$\|\rho_0\|_{L^\infty(\Omega)} 
\in ]0,1/\vartheta[$
with $\vartheta=C_{gn}C_{poi}\max\{C^{reg}_1,C^{reg}_2\},$
where $C^{reg}_1$ and $C^{reg}_2$ the regularity constants defined 
on the proof of Lemma~\ref{eq:lem:estimates3_full}
( see \eqref{eq:stokes_for_uuu} and \eqref{eq:regularity_for_Lw}) 
and $C_{gn},C_{poi}$ are the Poincar\'e an Gagliardo-Nirenberg constants 
(see \eqref{eq:gagnir} and \eqref{eq:poi}), respectively.

\item[(H$_1$)] The initial  density $\rho_0$  belongs to $W^{1,q}(\Omega)$
for $q>3$
and $\rho_0(\bx)\in [\alpha,\beta]\subset]0,\infty[$ a.e. in~$\Omega$,
\item[(H$_2$)] The initial velocity $\bu_0$  belongs to $D(A)$ and
\item[(H$_3$)] The initial angular velocity $\bw_0$   belongs to $D(L),$
\item[(H$_4$)] The functions $\bm$ and $\bq$ belong to $C^1([0,T];\bH^2(\Omega))$,
\end{enumerate}
are satisfied.
Then, there exists $\upkappa_j\in\R^+$ for $j=1,\ldots,11$ 
depending only on $\Omega,c_a,c_0,c_d,\alpha,\beta,\mu_r,\bm$ and $\bq$
(independents of $f$ and $g$) and 
two small enough times $T_1,T_2\in [0,T_*]$, 
independents of $f$ and $g$, such that the following estimates hold
\begin{eqnarray}
&&\| \rho\|_{L^\infty([0,T_*];L^\infty(\Omega))}\in ]\alpha,\beta[,
	\label{eq:lema:aprioriestimatesforrho}\\
&&\| \bu\|_{L^\infty([0,T_1];\bH^1_0(\Omega))}+
	\| \bw\|_{L^\infty([0,T_1];\bH^1_0(\Omega))}
\nonumber
\\
&&
\hspace{3.1cm}
	+\|\bu_t\|_{L^\infty([0,T_1];\bL^2(\Omega))}+
	\| \bw_t\|_{L^\infty([0,T_1];\bL^2(\Omega))}
	\le\upkappa_1,\qquad
\label{teo:global_estimates_kapa1}
\\
&&
\| \bu_t\|_{L^\infty([0,T_2];\bL^2(\Omega))}+
	\| \bw_t\|_{L^\infty([0,T_2];\bL^2(\Omega))}+
	\|\bu_t\|_{L^2([0,T_2];\bH^1(\Omega))}
	\nonumber
\\
&&
\hspace{3.1cm}
	+
	\| \bw_t\|_{L^2([0,T_2];\bH^1(\Omega))}+
	\| \nabla \rho\|_{L^\infty([0,T_*];\bL^q(\Omega))}
	\le\upkappa_2,
\label{teo:global_estimates_kapa2}
\\
&&
	\| h\|_{L^\infty([0,T_2];H^2(\Omega))}\le\upkappa_3,
\label{teo:global_estimates_kapa3_4}
\\
&&
	\| \nabla h_t\|_{L^\infty([0,T_2];\bL^2(\Omega))}\le\upkappa_4+\upkappa_5 
	\|(f,g)\|_{[H^1(0,T)]^2},
\label{teo:global_estimates_kapa5_8}
\\
&&
\| \bu\|_{L^\infty([0,T_1];\bH^2(\Omega))}+
	\| p\|_{L^\infty([0,T_1];H^1(\Omega))}
	\le\upkappa_6+\upkappa_{7} \|(f,g)\|_{[H^1(0,T)]^2},
\label{teo:global_estimates_kapa9_10}
\\
&&
\| \rho_t\|_{L^\infty([0,T_*];L^2(\Omega))}\le \upkappa_{8},
	\qquad
	\| \rho_t\|_{L^\infty([0,T_*];L^q(\Omega))}
	\le\upkappa_{9}+\upkappa_{10} \|(f,g)\|_{[H^1(0,T)]^2},
\label{teo:global_estimates_kapa11_13}
\\
&&
\| \bu\|_{L^2([0,T_2];\bW^{2,s}(\Omega))}+
	\| p\|_{L^2([0,T_2];W^{1,s}(\Omega))}
	\le \upkappa_{11},
\label{teo:global_estimates_kapa14}
\end{eqnarray}
where $q\in ]3,\infty[$ and $s\in [2,\infty[$.
\end{theorem}

We remark that some a priori estimates 
were introduced in \cite{boldrini_2003} in order to prove
the existence of weak solutions. Now, 
in Theorem~\ref{teo:global_estimates} we include some new estimates
and also in the proofs of the existing estimates
we introduce a different methodology.

\section{Proof of main result}
\label{subsec:a_priori_estimates}

In order to prove Theorem~\ref{teo:global_estimates} we recall some 
useful results. Then, we prove some Lemmas and finally we present the 
details of the proof.

\subsection{Some useful inequalities}
We recall some inequalities which will be used frequently
in order to get the desired estimates: 
\begin{itemize}
 \item[(i)] The Young's inequality. Let us consider $p,q\in ]1,\infty[$
 such that $p^{-1}+q^{-1}=1$, then 
\begin{eqnarray}
ab\le \epsilon a^p+C_{\epsilon} b^q,
\quad a,b\ge 0, 
\quad \epsilon>0,
\quad C_{\epsilon}=(p-1)\epsilon^{(1-q)}p^{-q}.
\label{eq:yung}
\end{eqnarray}
In particular, for $p=q=2$ and $\epsilon=1/2$ we have the Cauchy inequality.
 \item[(ii)] The Pioncar\'e  inequality. Let $\Omega$ be a connected,
bounded Lipschitz domain. Then, there exists 
$C_{poi}> 0$ depending only on  $p$ and $\Omega$ such that
\begin{eqnarray}
\|u\|_{L^q(\Omega)}\le C_{poi} \|\nabla u\|_{L^p(\Omega)},\quad
p\in [1,3[,\quad q\in [1,3p(3-p)^{-1}],\quad
u\in W^{1,p}_0(\Omega).
\label{eq:poi}
\end{eqnarray}
For a more general inequality in $W^{1,p}(\Omega)$ we refer
to proposition III.2.39 in~\cite{boyer_book}.

\item[(iii)] The Gagliardo-Nirenberg inequality. 
Let $\Omega\subset\R^3$ be a bounded  Lipschitz domain.
There exists $C_{gn}>0$ depending only on 
$q$ and $\Omega$ such that
\begin{eqnarray}
\|\nabla u\|_{L^{2q/(q-2)(\Omega)}}
\le C_{gn} \|\nabla u\|^{1-3/q}_{L^2(\Omega)}\|u\|^{3/q}_{H^2(\Omega)},
\quad q\in [2,\infty[,\quad 
u\in H^2(\Omega).
\label{eq:gagnir}
\end{eqnarray}
For a more general version of this inequality we refer to 
proposition III.2.35 in~\cite{boyer_book}.

\item[(iv)] Continuous embedding of $H^2(\Omega)$ in $L^\infty(\Omega)$.
Let $\Omega\subset\R^3$ be a bounded  Lipschitz domain. Then 
$H^2(\Omega)$ is continuous embedding in $L^\infty(\Omega)$ or
equivalently there exists $C^{2,\infty}_{iny}>0$ such that
\begin{eqnarray}
\| u\|_{L^\infty(\Omega)}\le C^{2,\infty}_{iny} \| u\|_{H^2(\Omega)},
\quad u\in H^2(\Omega).
\label{eq:h2linf}
\end{eqnarray}

\end{itemize}
For other inequalities which are not given in the previous list we
refer the books \cite{evans_book,boyer_book}.

\subsection{Space estimates}
In this section we obtain  some space estimates. Then, by notational convenience, 
we use  $\|\cdot\|_{L^p}$ and $\|\cdot\|_{H^p}$   
to abbreviate
the norms $\|\cdot\|_{{L^p}(\Omega)}$ and $\|\cdot\|_{{H^p}(\Omega)}$,
respectively.

\begin{lemma}
\label{lema:aprioriestimatesforrho}
The following estimate holds:
$0< \alpha \le \rho(\bx,t)\le\beta$ for 
all $(\bx,t)\in\Omega\times [0,T_*]$.
\end{lemma}

\begin{proof}
We deduce the estimate by equations \eqref{eq:incompresibilidad}
and \eqref{eq:ecuacion_continuidad}, the hypothesis (H$_1$)
and the maximum principle.
\end{proof}

\begin{lemma}
\label{lema:aprioriestimatesforhandr}
There exists  $\Uppi_i,$  $i=1,\ldots,6,$ 
depending only on $\Omega,\alpha,\beta,\bm$ and $\bq$
(independents of $f$ and $g$)
such that the following estimates holds
\begin{eqnarray}
&&\|\nabla h(\cdot,t)\|_{\bL^2}\le \Uppi_1,
\label{hx_and_rx_estimate}
\\
&&\|h (\cdot,t)\|_{H^2}
 \le 
 \Uppi_2+\Uppi_3 \|\nabla\rho(\cdot,t)\|^{q/(q-3)}_{\bL^q},
\label{h_and_r_estimate}
\\
&&\|\nabla h_t(\cdot,t)\|_{\bL^2}\le 
\Uppi_4
+
 \Big(\Uppi_5+\Uppi_6 \|\nabla\rho(\cdot,t)\|^{q/(q-3)}_{\bL^q}\Big)
 \|u(\cdot,t)\|_{\bH^2}
 \|\nabla\rho(\cdot,t)\|_{\bL^q},
 \qquad
\label{ht_and_rt_estimate}
\end{eqnarray}
for all $t\in [0,T_*]$  and $q\in ]3,\infty[ $.

\end{lemma}

\begin{proof}
From \eqref{eq:helmholtz_f_dos}, by applying 
Lemma~\ref{lema:aprioriestimatesforrho}, integration by parts, the boundary condition
\eqref{eq:helmholtz_f_tres}, and H\"older and Young inequalities, we
have that
\begin{eqnarray*}
\alpha\int_{\Omega}|\nabla h|^2(\bx,t)d\bx
&\le & \int_{\Omega}\Big(\rho|\nabla h|^2\Big)(\bx,t)d\bx
=\int_{\Omega}\Big(\rho\nabla h\cdot \bm\Big)(\bx,t)d\bx
\\
&\le & \beta\|\nabla h(\cdot,t)\|_{\bL^2}\|\bm(\cdot,t)\|_{\bL^2}
\le \frac{\epsilon\beta}{2}\|\nabla h(\cdot,t)\|^2_{\bL^2}
+\frac{\beta}{2\epsilon}\|\bm(\cdot,t)\|^2_{\bL^2},
\end{eqnarray*}
for each $t\in [0,T_*]$ and $\epsilon>0$. Hence,
selecting $\epsilon\in ]0,2\alpha\beta^{-1}[$ 
and defining 
\begin{eqnarray}
 \Uppi_1=\sqrt{\beta\epsilon^{-1} (2\alpha-\beta \epsilon)^{-1}}
 \|\bm(\cdot,t)\|^2_{\bL^2}.
\label{eq:uppi_uno}
\end{eqnarray}
we see that the estimate \eqref{hx_and_rx_estimate} is valid. 

Now, we can proceed to prove \eqref{h_and_r_estimate}. We start by recalling
the identities $\diver(\rho \nabla h)=\rho \Delta h+\nabla\rho\cdot\nabla h$ and
$\diver(\rho \bm)=\rho\diver(\bm)+\nabla\rho\cdot \bm$, which imply that 
the  equation  \eqref{eq:helmholtz_f_dos} can be rewritten as follows
\begin{eqnarray}
-\Delta h=-\diver(\bm)-\frac{1}{\rho}\nabla\rho\cdot \bm
 +\frac{1}{\rho}\nabla\rho\cdot\nabla h.
 \label{eq:elliptic_helmholtz}
\end{eqnarray}
Clearly, by application of  the estimate \eqref{hx_and_rx_estimate}, we deduce
that the right hand side of \eqref{eq:elliptic_helmholtz}
belongs to $L^2(\Omega)$. Then, by
the regularity of solutions for  elliptic equations 
(Theorem III.4.3\cite{boyer_book}) applied to
\eqref{eq:elliptic_helmholtz}-\eqref{eq:helmholtz_f_tres},
the inequality \eqref{eq:gagnir},
the trace thorem (Theorem III.2.19\cite{boyer_book}),
Lemma~\ref{lema:aprioriestimatesforrho}, the inequality \eqref{eq:h2linf},
and  the estimate \eqref{hx_and_rx_estimate}, 
we can follow that there exists $C_{reg}>0$, independent of $h$ such that 
the following bound
\begin{eqnarray*}
&& \|h (\cdot,t)\|_{H^2}
 \le C_{reg}
 \left\{
  \left\|\left(\diver(\bm)+\frac{1}{\rho}\nabla\rho\cdot \bm
  -\frac{1}{\rho}\nabla\rho\cdot\nabla h\right)(\cdot,t)\right\|_{\bL^2}
  +\|\bm\cdot\bn\|_{H^{1/2}}
  \right\}
  \\
  &&\hspace{1cm}
  \le C_{reg}\Big\{
  \|\nabla \bm(\cdot,t)\|_{\bL^2}
  +\frac{1}{\alpha}\| \bm(\cdot,t)\|_{L^{\infty}}\|\nabla\rho(\cdot,t)\|_{\bL^2}
  \\
  &&\hspace{2.5cm}  
  +\frac{1}{\alpha}\|\nabla\rho(\cdot,t)\|_{\bL^q}\|\nabla h(\cdot,t)\|_{\bL^{2q/(q-2)}}
  +C_{tr}\|\bm\cdot\bn\|^{1/2}_{L^{2}}\|\bm\cdot\bn\|^{1/2}_{H^{1}}
  \Big\}
  \\
  &&\hspace{1cm}
  \le
  C_{reg}\Big\{
  \|\nabla \bm(\cdot,t)\|_{\bL^2}
  +\frac{|\Omega|^{\frac{q-2}{2q}}}{\alpha}\| \bm(\cdot,t)\|_{\bL^{\infty}}
  \|\nabla\rho(\cdot,t)\|_{\bL^q}
  \\
  &&\hspace{2.5cm}
  +\frac{C_{gn}}{\alpha}\|\nabla\rho(\cdot,t)\|_{\bL^q}\|
     \|\nabla h(\cdot,t)\|^{1-3/q}_{\bL^2}\|h(\cdot,t)\|^{3/q}_{H^2}
     +C_{tr}\|\bm\|_{\bH^{1}}
     \Big\}
  \\
  &&\hspace{1cm}
  \le
  C_{reg}\Big\{
 (C_{tr}+1)\|\bm\|_{\bH^{1}}
  +\frac{C^{2,\infty}_{iny}|\Omega|^{\frac{q-2}{2q}}}{\alpha}
  \| \bm(\cdot,t)\|_{\bH^2}
  \|\nabla\rho(\cdot,t)\|_{\bL^q}
  \\
  &&\hspace{2.5cm}
  +\frac{C_{gn}(\Uppi_1)^{1-3/q}}{\alpha}\;
  \|\bm(\cdot,t)\|^{1-3/q}_{\bL^2}
  \|\nabla\rho(\cdot,t)\|_{\bL^q}\|h(\cdot,t)\|^{3/q}_{H^2}\Big\},
\end{eqnarray*}
holds for each $t\in [0,T_*]$ and $q\in ]3,\infty[ $. 
Here $C_{tr}$ denotes the trace constant.
Thus,
selecting $\epsilon'\in ]0,\alpha(C_{reg}C_{gn})^{-1}[$
and $\epsilon''>0$,
by the application
of two times of the Young's inequality \eqref{eq:yung} 
we complete the proof
of \eqref{h_and_r_estimate} with $\Uppi_2$ and $\Uppi_3$
given by 
\begin{align}
&\Uppi_2=\frac{C_{reg}}{\alpha-\epsilon'C_{reg}C_{gn}}\Big(
    \alpha(C_{tr}+1)\|\bm(\cdot,t)\|_{\bH^1}
    +|\Omega|^{(q-2)/{2q}}C^{2,\infty}_{iny}\epsilon''\|\bm(\cdot,t)\|^{3/q}_{\bH^2}
    \Big),
    \hspace{1cm}
\label{eq:uppi_dos}    
 \\
&\Uppi_3=\frac{C_{reg}}{\alpha-\epsilon'C_{reg}C_{gn}}\Big(
       |\Omega|^{(q-2)/{2q}}C^{2,\infty}_{iny}C_{\epsilon''}
        +C_{gn}C_{\epsilon'}\Uppi_1\Big).
\label{eq:uppi_tres}
\end{align}

The proof of \eqref{ht_and_rt_estimate} is given as follows.
Taking $\partial_t$ to 
the first equation of \eqref{eq:helmholtz_f_dos}, testing the result by $h_t$,
using the estimate of Lemma~\ref{lema:aprioriestimatesforrho},
the H\"older inequality, the equation \eqref{eq:ecuacion_continuidad}
and the inequality \eqref{eq:gagnir},
we have that
\begin{eqnarray*}
&&\alpha\|\nabla h_t(\cdot,t)\|^2_{\bL^2}
\le 
 \int_{\Omega} \rho|\nabla h_t (\cdot,t)|^2d\bx
 \\
&& \qquad
 \le 
 \left|\int_{\Omega} \Big(\rho \bm_t \cdot \nabla h_t\Big)(\cdot,t)d\bx\right|
 +\left|\int_{\Omega} \Big(\rho_t \nabla h\cdot \nabla h_t\Big)(\cdot,t)d\bx\right|
 +\left|\int_{\Omega} \Big(\rho_t   \bm\cdot \nabla h_t\Big)(\cdot,t)d\bx\right|
 \\
&& \qquad
 \le 
 \|\rho(\cdot,t)\|_{\bL^\infty}\|\bm_t(\cdot,t)\|_{\bL^2}
 \|\nabla h_t(\cdot,t)\|_{\bL^2}
 +\left|\int_{\Omega} \Big((\bu\cdot\nabla\rho) 
 \nabla h\cdot \nabla h_t\Big)(\cdot,t)d\bx\right|
 \\
 &&\qquad\quad
 +\left|\int_{\Omega} \Big((\bu\cdot\nabla\rho)
      \bm\cdot \nabla h_t\Big)(\cdot,t)d\bx\right|
 \\ 
 &&\qquad
  \le 
 \beta\|\bm_t(\cdot,t)\|_{\bL^2}\|\nabla h_t(\cdot,t)\|_{\bL^2}
 +\| \bu(\cdot,t)\|_{\bL^\infty}\|\nabla\rho(\cdot,t)\|_{\bL^q}
 \|\nabla h(\cdot,t)\|_{\bL^{2q/(q-2)}}\|\nabla h_t(\cdot,t)\|_{\bL^2}
  \\ 
 &&\qquad\quad
 +\| \bm(\cdot,t)\|_{\bL^\infty}\|\nabla\rho(\cdot,t)\|_{\bL^q}
 \|\bu(\cdot,t)\|_{\bL^{2q/(q-2)}}\|\nabla h_t(\cdot,t)\|_{\bL^2}
 \\
&& \qquad
  \le 
 \|\nabla h_t(\cdot,t)\|_{\bL^2}\Big\{\beta\|\bm_t(\cdot,t)\|_{\bL^2}
 +C_{gn}\| \bu(\cdot,t)\|_{\bL^\infty}\|\nabla\rho(\cdot,t)\|_{\bL^q}
 \|\nabla h(\cdot,t)\|^{1-3/q}_{\bL^{2}}\|h(\cdot,t)\|^{3/q}_{H^{2}}
  \\ 
&& \qquad\quad
 +\| \bm(\cdot,t)\|_{\bL^\infty}\|\nabla\rho(\cdot,t)\|_{\bL^q}
 \|\bu(\cdot,t)\|_{\bL^{2q/(q-2)}}\Big\}.
\end{eqnarray*}
for $q\in ]3,\infty[.$
Then, by \eqref{eq:poi},  \eqref{eq:h2linf}, \eqref{eq:yung},
\eqref{hx_and_rx_estimate} and
\eqref{h_and_r_estimate}, we obtain
 \begin{eqnarray*}
\|\nabla h_t(\cdot,t)\|_{\bL^2} 
  \le 
\frac{\beta}{\alpha}\|\bm_t(\cdot,t)\|_{\bL^2}
 +\frac{C_{gn} C^{2,\infty}_{iny}\dot{\epsilon}}{\alpha}
 \| \bu(\cdot,t)\|_{\bH^2}\|\nabla\rho(\cdot,t)\|_{\bL^q}
 \Big\{\Uppi_2+\Uppi_3 \|\nabla\rho(\cdot,t)\|_{\bL^{q/(q-3)}}\Big\}
 \\
 \hspace{1cm}
 +\frac{C_{gn} C^{2,\infty}_{iny}}{\alpha}
 \Big( C_{\dot{\epsilon}}\Uppi_1+C_{poi}\Big)
 \| \bm(\cdot,t)\|_{\bH^2}\|\nabla \bu(\cdot,t)\|_{\bL^2}
 \|\nabla\rho(\cdot,t)\|_{\bL^q},
\end{eqnarray*}
for $\dot{\epsilon}>0$, 
which implies \eqref{ht_and_rt_estimate} by 
defining 
\begin{align}
&\Uppi_4=\frac{\beta}{\alpha}\|\bm_t(\cdot,t)\|_{\bL^2},
\label{eq:uppi_cuatro}
\\
&\Uppi_5=
\frac{C_{gn}C^{2,\infty}_{iny}}{\alpha}
 \Big(\dot{\epsilon}\Uppi_2+
 \Big(C_{\dot{\epsilon}}\Uppi_1+C_{poi}\Big)
 \|\bm(\cdot,t)\|_{\bH^2}
 \Big),
 \label{eq:uppi_cinco}
 \\
&\Uppi_6=
\frac{C_{gn}C^{2,\infty}_{iny}}{\alpha}
 \;\dot{\epsilon}\;\Uppi_3.
\label{eq:uppi_seis}
\end{align}

Finally, by hypothesis (H$_4$) and \eqref{eq:uppi_uno}, \eqref{eq:uppi_dos}-\eqref{eq:uppi_seis}
we note that $\Uppi_i$ are well defined and also are independents of $f$ and $g$
and conclude the proof of lemma.
\end{proof}

\begin{lemma}
\label{prop:auxiliar_stokes_elliptic_results}
There exists  $\Upsilon_i\in \R^+$ for $i\in\{1,\ldots,5\}$,
depending only on $\Omega,c_a,c_0,c_d,\alpha,\beta,\mu_r,\bm$ and $\bq$
(independents of $f$ and $g$), such that the following estimate holds:
\begin{eqnarray} 
&&\Upsilon_1\|\bu(\cdot,t)\|_{\bH^2}+
\|p(\cdot,t)\|_{H^1}+\Upsilon_1\|\bw(\cdot,t)\|_{\bH^2}
\nonumber
\\
&&
\quad
\le
\Upsilon_2
\Big[\|\nabla \bu(\cdot,t)\|^{3}_{\bL^2}+\|\nabla \bw(\cdot,t)\|^{3}_{\bL^2}\Big]
+
\Upsilon_3
\Big[\|\nabla \bu(\cdot,t)\|_{\bL^2}+\|\nabla \bw(\cdot,t)\|_{\bL^2}\Big]
\nonumber
\\
&&
\qquad
+\Upsilon_4
\Big[|f(t)|+|g(t)|\Big]
+\Upsilon_5
\Big[
\|\big(\sqrt{\rho} \bu_t\big)(\cdot,t)\|_{\bL^2}
+\|\big(\sqrt{\rho} \bw_t\big)(\cdot,t)\|_{\bL^2}
\Big]
,
\label{eq:regularity_u_p_w}
\end{eqnarray}
for all $t\in[0,T_*]$.
\end{lemma}

\begin{proof}
The inequality  \eqref{eq:regularity_u_p_w} is a consequence of the regularity
of solutions for the Stokes system satisfied by $\bu$ and $p$ and the uniformly
elliptic  equation satisfied by $\bw$. Indeed, we first note that
the equations \eqref{eq:momento_lineal},
\eqref{eq:incompresibilidad} and \eqref{eq:helmholtz_f_uno}
imply that $\bu$ and $p$ satisfy the Stokes problem given by the equation
\begin{eqnarray}
-(\mu +\mu_r)\Delta \bu+\nabla p
	=2\mu_r \curl \bw+\rho f(\nabla h-\bm)
	-\rho \bu_t-\rho (\bu\cdot\nabla )\bu,
	\quad\mbox{in}\quad Q_T,
\label{eq:stokes_problem}
\end{eqnarray}
where the incompressibility condition is given by  \eqref{eq:incompresibilidad}
and the initial and boundary conditions are given by
\eqref{eq:direct_problem_ic} and \eqref{eq:direct_problem_bc}, respectively. 
Hence, by applying the result given in \cite{temam_1977}
for the regularity of the 
solutions for stokes equation, 
the Minkowski and H\"older  inequalities, we deduce that 
\begin{eqnarray}
&& \|\bu(\cdot,t)\|_{\bH^2}+\|p(\cdot,t)\|_{H^1}
\nonumber
\\
&&
\qquad
\le C^{reg}_1\; \Big[
2\mu_r \|\curl \bw(\cdot,t)\|_{\bL^2}
+\|\big(\rho f(\nabla h-\bm)\big)(\cdot,t)\|_{\bL^2}
+\|(\rho \bu_t)(\cdot,t)\|_{\bL^2}
\nonumber
\\
&&
\qquad\quad
+\|\rho ((\bu\cdot\nabla )\bu)(\cdot,t)\|_{\bL^2}\Big]
\nonumber
\\
&&
\qquad
\le C^{reg}_1\; \Big[
2\mu_r \|\nabla \bw(\cdot,t)\|_{\bL^2}
+|f(t)|\|\rho (\cdot,t)\|_{L^2}
\Big(\|\nabla h(\cdot,t)\|_{\bL^2}+\|\bm(\cdot,t)\|_{\bL^2}\Big)
\nonumber
\\
&&
\qquad\quad
+\|(\rho \bu_t)(\cdot,t)\|_{\bL^2}
+\|\rho(\cdot,t)\|_{L^\infty}\|\bu(\cdot,t)\|_{\bL^6}\|\nabla \bu(\cdot,t)\|_{\bL^3}\Big],
\label{eq:stokes_for_uuu}
\end{eqnarray}
where $C^{reg}_1$ is a positive constant depending on $\mu,\mu_r$
and $\Omega$. In the second place,
by \eqref{eq:momento_angular}, \eqref{eq:helmholtz_f_uno} and \eqref{eq:def_of_L_and_L0},
we deduce that $\bw$ satisfies the following equation
\begin{eqnarray}
L\bw
	=2\mu_r \curl \bu+\rho g \bq
	-\rho \bw_t-\rho (\bu\cdot\nabla )\bw,
	\quad\mbox{in}\quad Q_T.
\end{eqnarray}
Then by the regularity results for  the 
solutions for uniformly elliptic  equations (see for instance \cite{evans_book}), 
the Minkowski and H\"older  inequalities, and \eqref{eq:poi}, we have that 
\begin{eqnarray}
 \|\bw(\cdot,t)\|_{\bH^2}
&\le& C^{reg}_2\; \Big[
2\mu_r \|\curl \bu(\cdot,t)\|_{\bL^2}
+\|\big(\rho g\bq\big)(\cdot,t)\|_{\bL^2}
+\|(\rho \bw_t)(\cdot,t)\|_{\bL^2}
\nonumber
\\
&&
+\|\rho ((\bw\cdot\nabla )\bw)(\cdot,t)\|_{\bL^2}
+\|\bw(\cdot,t)\|_{\bL^2}\Big]
\nonumber
\\
&\le& C^{reg}_2\; \Big[
2\mu_r \|\nabla \bu(\cdot,t)\|_{\bL^2}
+|g(t)|\|\rho (\cdot,t)\|_{L^2}\|\bq(\cdot,t)\|_{\bL^2}
+\|(\rho \bw_t)(\cdot,t)\|_{\bL^2}
\nonumber
\\
&&
+\|\rho(\cdot,t)\|_{L^2}\|\bu(\cdot,t)\|_{\bL^6}\|\nabla\bw(\cdot,t)\|_{\bL^3}
+C_{poi}\|\nabla \bw(\cdot,t)\|_{\bL^2}\Big],
\label{eq:regularity_for_Lw}
\end{eqnarray}
where $C^{reg}_2$ is a positive constant depending only on $\Omega$ and 
on the coefficients of $L$.
Now, we note that the second terms on the right hand sides 
of \eqref{eq:stokes_for_uuu}
and \eqref{eq:regularity_for_Lw} can be bound
by application of 
Lemmas~\ref{lema:aprioriestimatesforrho}-\ref{lema:aprioriestimatesforhandr}
and \eqref{eq:gagnir}.
Hence, if we sum the bounded results,  we
obtain the following inequality
\begin{eqnarray}
&& \|\bu(\cdot,t)\|_{\bH^2}+\|p(\cdot,t)\|_{H^1}+\|\bw(\cdot,t)\|_{\bH^2}
\nonumber
\\&&
\qquad
\le
\max\{C^{reg}_1,C^{reg}_2\}\;\Big\{
(2\mu_r +C_{poi}) \;\|\nabla \bw(\cdot,t)\|_{\bL^2}+2\mu_r\|\nabla \bu(\cdot,t)\|_{\bL^2}
\nonumber
\\&&
\qquad\quad
+\beta |\Omega|^{1/2}
\;
\Big[|f(t)|\Big(\Uppi_1+\|\bm(\cdot,t)\|_{\bL^2}\Big)+|g(t)|\|\bq(\cdot,t)\|_{\bL^2}\Big]
\nonumber
\\&&
\qquad\quad
+
(\beta)^{1/2}
\Big[
\|\big(\sqrt{\rho} \bu_t\big)(\cdot,t)\|_{\bL^2}
+\|\big(\sqrt{\rho} \bw_t\big)(\cdot,t)\|_{\bL^2}
\Big]
+
\beta |\Omega|^{1/2}C_{gn}C_{poi}\times
\nonumber
\\&&
\qquad\qquad\quad
\Big[\|\nabla \bu(\cdot,t)\|^{3/2}_{\bL^2}\|\bu(\cdot,t)\|^{1/2}_{\bH^2}
+\|\nabla \bu(\cdot,t)\|_{\bL^2}\|\nabla \bw(\cdot,t)\|^{1/2}_{\bL^2}
\|\bw(\cdot,t)\|^{1/2}_{\bH^2}\Big]\Big\},
\hspace{1cm}
\label{eq:regularity_u_p_w_prev}
\end{eqnarray}
for all $t\in [0,T_*]$. 
Now, denoting $C_M=\max\{C^{reg}_1,C^{reg}_2\}\beta |\Omega|^{1/2}C_{gn}C_{poi}$, 
for $\epsilon^*\in\mathbb{R}^+$ 
 we define $\Upsilon_i$ for $i=1,\ldots,5$
as follows
\begin{eqnarray}
&&\Upsilon_1=(1-\epsilon^*C_M),
\quad
\Upsilon_2=(C_M+1)\;(2\epsilon^*)^{-1},
\label{eq:upsilon2}
\\&&
\Upsilon_3=(2\mu_r +C_{poi})\max\{C^{reg}_1,C^{reg}_2\},
\nonumber\\&&
\Upsilon_4=\beta |\Omega|^{1/2}\max\{C^{reg}_1,C^{reg}_2\}
\max\{\Uppi_1+\|\bm(\cdot,t)\|_{\bL^2}\;,\;\|\bq(\cdot,t)\|_{\bL^2}\},
\nonumber\\&&
\Upsilon_5=(\beta)^{1/2}\max\{C^{reg}_1,C^{reg}_2\}
\label{eq:upsilon5}
\end{eqnarray}
Now, let us consider the notation $\varpi=\|\rho_0\|_{\infty}\vartheta$
with $\vartheta$ defined in (H$_0$). Then,
by (H$_0$) we note that $C_M=\varpi|\Omega|^{1/2}$ 
and $]0,(1-\varpi)/C_M[\subset]0,1/C_M[$.
Thus, selecting $\epsilon^*\in ]0,(1-\varpi)/C_M[ $
and applying the Cauchy inequality with $\epsilon^*$ to the last two terms
of \eqref{eq:regularity_u_p_w_prev} and \eqref{eq:yung} with $p=3/2,q=3$ and $\epsilon=1$
to the product $\|\nabla \bu(\cdot,t)\|^2_{\bL^2}\|\nabla \bw(\cdot,t)\|_{\bL^2}$
we get~\eqref{eq:regularity_u_p_w}.
\end{proof}

\begin{remark}
\label{rem:epsilon}
In order to prove that $\Upsilon_1>0$ is enough to select $\epsilon^*\in ]0,1/C_M[$.
However, in the proof of Lemma~\ref{lem:u_ut_w_wt_estimate} we will need that   
$\epsilon^*\in ]0,(1-\varpi)/C_M[ $.
\end{remark}

\begin{lemma}
\label{lem:u_ut_w_wt_estimate}
There exists  $\Upxi_1,\Upxi_2,\Upxi_3\in \R^+$,
depending only on $\Omega,c_a,c_0,c_d,\alpha,\beta,\mu_r,\bm$ and $\bq$
(independents of $f$ and $g$), such that the following estimate holds:
\begin{eqnarray}
&&\frac{d}{dt}\Big(\|\nabla \bu(\cdot,t)\|^{2}_{\bL^2}
+\|\nabla \bw(\cdot,t)\|^{2}_{\bL^2}\Big)
+\|(\sqrt{\rho} \bv_t)(\cdot,t)\|^{2}_{\bL^2}
+\|(\sqrt{\rho} \bw_t)(\cdot,t)\|^{2}_{\bL^2}
\nonumber\\
&&
\hspace{1cm}
\le\Upxi_1
\Big[\|\nabla \bu(\cdot,t)\|^{6}_{\bL^2}
+\|\nabla \bw(\cdot,t)\|^{6}_{\bL^2}\Big]
\nonumber\\
&&
\hspace{1.3cm}
+\Upxi_2
\Big[\|\nabla \bu(\cdot,t)\|^{2}_{\bL^2}
+\|\nabla \bw(\cdot,t)\|^{2}_{\bL^2}\Big]
+\Upxi_3\Big[|f(t)|^{2}+|g(t)|^{2}\Big],
\hspace{1cm}
\label{lem:u_ut_w_wt_main_estimate}
\end{eqnarray}
for $t\in [0,T_{*}]$.
In particular,
there exists $T_1\in [0,T_*]$  and $\Uptheta:[0,T_1]\to\R^+$ independents of $f$ and $g$ such that
the following estimate holds 
\begin{eqnarray}
\|\bu(\cdot,t)\|_{\bH^1_0}+
\|\bw(\cdot,t)\|_{\bH^1_0}\le \Uptheta(t),
\label{eq:lem:u_ut_w_wt_estimate}
\end{eqnarray}
 for all $t\in [0,T_1].$
\end{lemma}

\begin{proof}
Testing the equations \eqref{eq:momento_lineal} and \eqref{eq:momento_angular} by $\bu_t$ 
and $\bw_t$, respectively;  
 summing the results; and applying the Minkowski and H\"older inequalities, we get
 \begin{eqnarray}
&&\frac{(\mu+\mu_r)}{2}\frac{d}{dt}\int_{\Omega}|\nabla \bu(\bx,t)|^2d\bx
+\frac{(c_0+2c_d)}{2}\frac{d}{dt}\int_{\Omega}|\nabla \bw(\bx,t)|^2d\bx
+\int_{\Omega}\Big(\rho |\bu_t|^2\Big)(\bx,t)d\bx
\nonumber\\&&
\qquad
+\int_{\Omega}\Big(\rho |\bw_t|^2\Big)(\bx,t)d\bx
=-\int_{\Omega}\Big(\rho (\bu\cdot\nabla )\bu\cdot\bu_t\Big)(\bx,t) d\bx
+2\mu_r\int_{\Omega}\Big(\curl \bw\,\cdot \bu_t\Big)(\bx,t) d\bx
\nonumber\\&&
\qquad
+ f(t)\int_{\Omega}\Big(\rho(\nabla h-\bm)\cdot\bu_t\Big)(\bx,t) d\bx
-\int_{\Omega}\Big(\rho (\bu\cdot\nabla )\bw\cdot\bw_t\Big)(\bx,t) d\bx
\nonumber\\&&
\qquad
+2\mu_r\int_{\Omega}\Big(\curl \bu\;\cdot \bw_t \Big)(\bx,t) d\bx+ g(t)
\int_{\Omega}\Big(\rho\bq\cdot\bw_t\Big)(\bx,t) d\bx
-4\mu_r\int_{\Omega}\Big(\bw\cdot\bw_t\Big)(\bx,t) d\bx
\nonumber\\&&
\qquad
\le
\Big[
\|(\rho \bu_t)(\cdot,t)\|_{\bL^2}\|\bu(\cdot,t)\|_{\bL^6}\|\nabla \bu(\cdot,t)\|_{\bL^3}
	+\|(\rho \bw_t)(\cdot,t)\|_{\bL^2}\|\bu(\cdot,t)\|_{\bL^6}
	\|\nabla \bw(\cdot,t)\|_{\bL^3}
\Big]
\nonumber\\&&
\qquad\quad
	+\Big[
	2\mu_r\Big\{
	\|\nabla \bw(\cdot,t)\|_{\bL^2} \|\bu_t (\cdot,t)\|_{\bL^2}
	+\|\nabla \bu(\cdot,t)\|_{\bL^2} \|\bw_t (\cdot,t)\|_{\bL^2}\Big\}
	\Big]
\nonumber\\&&
\qquad\quad
	+\Big[
	|f(t)|\|(\rho \bu_t)(\cdot,t)\|_{\bL^2}
	\Big(\|\nabla h(\cdot,t)\|_{\bL^2}+\|\bm(\cdot,t)\|_{\bL^2}\Big)
	+
	|g(t)|\|(\rho \bw_t)(\cdot,t)\|_{\bL^2}\|\bq(\cdot,t)\|_{\bL^2}
	\Big]
\nonumber\\&&
\qquad\quad
	+\Big[
	4\mu_r\|\bw(\cdot,t)\|_{\bL^2}\|\bw_t(\cdot,t)\|_{\bL^2}
	\Big]
	:=\sum_{i=1}^4J_i,
\label{eq:ml_and_ma_times_ut_wt}
\end{eqnarray}
where each $J_i$ are defined by the corresponding brackets $\big[\quad\big]$.
Now, we will prove the estimate by getting  
some bounds for each $J_i$ and then applying a Gronwall type inequality.

The bound for $J_1$. By Lemma~\ref{lema:aprioriestimatesforrho}, 
inequalities \eqref{eq:poi}, \eqref{eq:gagnir}, \eqref{eq:yung},  and 
Lemma~\ref{prop:auxiliar_stokes_elliptic_results}, we deduce that
\begin{eqnarray}
J_1
&\le&
\sqrt{\beta} C_{poi} C_{gn}
\Big[
\|(\sqrt{\rho} \bu_t)(\cdot,t)\|_{\bL^2}\|\nabla \bu(\cdot,t)\|^{3/2}_{\bL^2}
\| \bu(\cdot,t)\|^{1/2}_{\bH^2}
\nonumber
\\&&
\hspace{2cm}
+\|(\sqrt{\rho} \bw_t)(\cdot,t)\|_{\bL^2}
\|\nabla \bu(\cdot,t)\|_{\bL^2}
\|\nabla \bw(\cdot,t)\|^{1/2}_{\bL^2}\| \bw(\cdot,t)\|^{1/2}_{\bH^2}\Big]
\nonumber
\\
&\le& 
\frac{\sqrt{\beta} C_{poi} C_{gn}}{2}\Big\{
\|(\sqrt{\rho} \bu_t)(\cdot,t)\|_{\bL^2}\;
\Big[\|\nabla \bu(\cdot,t)\|^{3}_{\bL^2}+\| \bu(\cdot,t)\|_{\bH^2}\Big]
\nonumber
\\&&
\hspace{2cm}
+
\|(\sqrt{\rho} \bw_t)(\cdot,t)\|_{\bL^2}
\Big[\|\nabla  \bu(\cdot,t)\|^{3}_{\bL^2}+
\|\nabla  \bw(\cdot,t)\|^{3}_{\bL^2}+\| \bw(\cdot,t)\|_{\bH^2}\Big]
\Big\}
\nonumber
\\
&\le&
\Big(\Uppsi_1 \overline{\epsilon}+\Uppsi_2\Big)
\Big[\|(\sqrt{\rho} \bu_t)(\cdot,t)\|^{2}_{\bL^2}
+\|(\sqrt{\rho} \bw_t)(\cdot,t)\|^{2}_{\bL^2}\Big]
\nonumber
\\&&
\qquad
+\Uppsi_3
\Big[\|\nabla \bu(\cdot,t)\|^{6}_{\bL^2}
+\|\nabla \bw(\cdot,t)\|^{6}_{\bL^2}\Big]
+\Uppsi_4
\Big[\|\nabla \bu(\cdot,t)\|^{2}_{\bL^2}
+\|\nabla \bw(\cdot,t)\|^{2}_{\bL^2}\Big]
\nonumber
\\&&
\qquad
+\Uppsi_5\Big[|f(t)|^{2}+|g(t)|^{2}\Big],
\label{eq:inequality_J1}
\end{eqnarray}
where $\overline{\epsilon}>0$ and
\begin{eqnarray}
&&\Uppsi_1=\sqrt{\beta} C_{poi} C_{gn}
	\Big(1+2\frac{\Upsilon_2}{\Upsilon_1}+\frac{\Upsilon_3}{\Upsilon_1}
	+\frac{\Upsilon_4}{\Upsilon_1}\Big),
\quad
\Uppsi_2=\sqrt{\beta} C_{poi} C_{gn}\frac{\Upsilon_5}{\Upsilon_1},
\label{eq:upsi2}
\\&&
\Uppsi_3=\frac{\sqrt{\beta} C_{poi} C_{gn}}{4\overline{\epsilon}}
	\Big(1+2\frac{\Upsilon_2}{\Upsilon_1}\Big),
	\quad
	\Uppsi_4=\frac{\sqrt{\beta} C_{poi} C_{gn}}{4\overline{\epsilon}}
	\frac{\Upsilon_3}{\Upsilon_1},
\nonumber\\&&
\Uppsi_6=\frac{\sqrt{\beta} C_{poi} C_{gn}}{4\overline{\epsilon}}
\frac{\Upsilon_4}{\Upsilon_1}.
\nonumber
\end{eqnarray}
Now, for $J_i$, $i=2,3,4$, by inequality \eqref{eq:poi},
Lemmas~\ref{lema:aprioriestimatesforrho}-\ref{lema:aprioriestimatesforhandr} 
and Young inequality, we deduce that
\begin{eqnarray}
&&J_2\le
\frac{4\mu_r\overline{\epsilon}}{\alpha}\Big[
\|(\sqrt{\rho} \bu_t)(\cdot,t)\|^{2}_{\bL^2}
	+
	\|(\sqrt{\rho} \bw_t)(\cdot,t)\|^{2}_{\bL^2}\Big]
\nonumber
\\&&
\qquad
	+\frac{\mu_r}{\overline{\epsilon}}
	\Big[\|\nabla \bu(\cdot,t)\|^{2}_{\bL^2}
	+\|\nabla \bw(\cdot,t)\|^{2}_{\bL^2}\Big],
\label{eq:inequality_J2}
\\&&
J_3\le
\sqrt{\beta}( \Uppi_1
	+\|\bm(\cdot,t)\|_{\bL^2})
	\|(\sqrt{\rho} \bu_t)(\cdot,t)\|_{\bL^2}|f(t)|
	+
	\sqrt{\beta}\|\bq(\cdot,t)\|_{\bL^2}
	\|(\sqrt{\rho} \bw_t)(\cdot,t)\|_{\bL^2}
	|g(t)|
\nonumber
\\&&
\hspace{0.5cm}
\le
\sqrt{\beta}\max\Big\{ \Uppi_1+\|\bm(\cdot,t)\|_{\bL^2}),\|\bq(\cdot,t)\|_{\bL^2}\Big\}
\overline{\epsilon}
\Big[\|(\sqrt{\rho} \bu_t)(\cdot,t)\|^2_{\bL^2}+\|(\sqrt{\rho} \bw_t)(\cdot,t)\|^2_{\bL^2}\Big]
\nonumber
\\&&
\hspace{0.5cm}
+\frac{\sqrt{\beta}}{4\overline{\epsilon}}
\max\Big\{ \Uppi_1+\|\bm(\cdot,t)\|_{\bL^2}),\|\bq(\cdot,t)\|_{\bL^2}\Big\}
\Big[|f(t)|^2+|g(t)|^2\Big],
\label{eq:inequality_J3}
\\&&
J_4\le
\frac{4\mu_r C_{poi}\overline{\epsilon}}{\alpha}
\|(\sqrt{\rho} \bw_t)(\cdot,t)\|^{2}_{\bL^2}
+
\frac{\mu_r C_{poi}}{\alpha\overline{\epsilon}}
\|\nabla \bw(\cdot,t)\|^{2}_{\bL^2}.
\label{eq:inequality_J4}
\end{eqnarray}
Inserting \eqref{eq:inequality_J1}-\eqref{eq:inequality_J4} in 
\eqref{eq:ml_and_ma_times_ut_wt} we get the estimate
\begin{eqnarray}
&&\frac{(\mu+\mu_r)}{2}\frac{d}{dt}\|\nabla \bu(\cdot,t)\|^2_{\bL^2}
+\frac{(c_0+2c_d)}{2}\frac{d}{dt}\|\nabla \bw(\cdot,t)\|^2_{\bL^2}
\nonumber\\
&&
\qquad
+
\Big(
1-\Uppsi_2-\Big\{\Uppsi_2+\frac{4\mu_r}{\alpha}
+\sqrt{\beta}\max\Big\{ \Uppi_1+\|\bm(\cdot,t)\|_{\bL^2}),\|\bq(\cdot,t)\|_{\bL^2}\Big\}
+\frac{4\mu_rC_{poi}}{\alpha}\Big\}
\overline{\epsilon}
\Big)\times
\nonumber\\
&&
\qquad
\Big[\|(\sqrt{\rho}\bu_t)(\cdot,t)\|^2_{\bL^2}
+\|(\sqrt{\rho}\bw_t)(\cdot,t)\|^2_{\bL^2}\Big]
\nonumber
\\&&
\qquad
\le
\Big(\Uppsi_3+\frac{\mu_rC_{poi}}{\alpha \overline{\epsilon}}\Big)
\Big[
\|\bu(\cdot,t)\|^6_{\bL^2}
+\|\bw(\cdot,t)\|^6_{\bL^2}
\Big]
\nonumber
\\&&
\qquad\quad
+
\Big(\Uppsi_4+\frac{\mu_r}{\overline{\epsilon}}
+\frac{\mu_rC_{poi}}{\alpha\overline{\epsilon}}\Big)
\Big[
\|\bu(\cdot,t)\|^2_{\bL^2}
+\|\bw(\cdot,t)\|^2_{\bL^2}
\Big]
\nonumber\\&&
\qquad\quad
+\Big(\Uppsi_4+\frac{1}{4\overline{\epsilon}}
\sqrt{\beta}\max\Big\{ \Uppi_1+\|\bm(\cdot,t)\|_{\bL^2}),\|\bq(\cdot,t)\|_{\bL^2}\Big\}
\Big)
	\Big[|f(t)|^2+|g(t)|^2\Big].
\label{eq:estimacionpenlemauh1wh1}
\end{eqnarray}
Now, by (H$_0$), \eqref{eq:upsilon2}, \eqref{eq:upsilon5}
and the fact that $\epsilon^*\in ]0,(1-\varpi)/C_M[ $ we have that
$\varpi<1-\epsilon^* \varpi |\Omega|^{1/2}<1$, which implies that
$\Upsilon_1>\sqrt{\beta}C_{gn}C_{poi}\Upsilon_5$ or equivalently, by
using the definition of $ \Uppi_2$ given on \eqref{eq:upsi2},
we have that $1-\Uppsi_2>0.$  Then, we can select $\overline{\epsilon}$
on the interval $]0,(1-\Uppsi_2){\mathcal{E}}^{-1}[$ with
\begin{eqnarray*}
\mathcal{E}
=
\Big\{\Uppsi_2+\frac{4\mu_r}{\alpha}
+\sqrt{\beta}\max\Big\{ \Uppi_1+\|\bm(\cdot,t)\|_{\bL^2},\|\bq(\cdot,t)\|_{\bL^2}\Big\}
+\frac{4\mu_rC_{poi}}{\alpha}\Big\}
\end{eqnarray*}
such that all coefficients in \eqref{eq:estimacionpenlemauh1wh1}
are positive. Thus defining
\begin{eqnarray*}
 \Upxi_1&=&\Big(\Uppsi_3+\frac{\mu_rC_{poi}}{\alpha \overline{\epsilon}}\Big)
 \mathcal{C}^{-1},
 \qquad
 \Upxi_2=
 \Big(\Uppsi_4+\frac{\mu_r}{\overline{\epsilon}}
+\frac{\mu_rC_{poi}}{\alpha\overline{\epsilon}}\Big)
 \mathcal{C}^{-1},
\\
\Upxi_3&=&
\Big(\Uppsi_4+\frac{1}{4\overline{\epsilon}}
\sqrt{\beta}\max\Big\{ \Uppi_1+\|\bm(\cdot,t)\|_{\bL^2},\|\bq(\cdot,t)\|_{\bL^2}\Big\}
\Big)
\mathcal{C}^{-1}
\end{eqnarray*}
with 
$
\mathcal{C}=\min\Big\{
2^{-1}(\mu+\mu_r),
2^{-1}(c_0+2c_d),
1-\Uppsi_2-\mathcal{E}
\overline{\epsilon}
\Big\}$
we deduce that the inequality
\eqref{lem:u_ut_w_wt_main_estimate} is valid. 

By the application of  Lemma~3 given on \cite{heywood_1980}, we deduce the
existence of $T_1\in [0,T_*]$ depending on $\|\nabla \bu(\cdot,0)\|_{\bL^2}$ and
$\|\nabla \bw(\cdot,0)\|_{\bL^2}$ such  the estimate
\eqref{eq:lem:u_ut_w_wt_estimate} holds
with $\Uptheta$ depending only on 
$\Upxi_1,\Upxi_2,\Upxi_3,\|\nabla \bu(\cdot,0)\|_{\bL^2}$ and 
$\|\nabla \bw(\cdot,0)\|_{\bL^2}$.
\end{proof}

\begin{lemma}
\label{lem:estimates3}
There exists $\Uppsi_i\in\R^+$, $i=1,\ldots,4,$ depending only on
$\Omega,c_a,c_0,c_d,\alpha,\beta,\mu_r,\bm$ and $\bq$
(independents of $f$ and $g$)
such that
the following estimate holds 
\begin{eqnarray}
&&\frac{d}{dt}\left(\|(\sqrt{\rho} \bu_t)(\cdot,t)\|^2_{\bL^2}
+\|(\sqrt{\rho} \bw_t)(\cdot,t)\|^2_{\bL^2}\right)
+\| \nabla \bu_t(\cdot,t)\|^2_{\bL^2}
+\| \nabla \bw_t(\cdot,t)\|^2_{\bL^2} 
\nonumber
\\&&
\qquad\quad
\le 
\Uppsi_1
\Big(
1+\max\Big\{\Big(\|\nabla h_t(\cdot,t)\|_{\bL^2}+1\Big)^2,1\Big\}
+\|\nabla\rho(\cdot,t)\|^{4q/(3q-6)}_{\bL^q}
\Big)
\Big[\|(\sqrt{\rho} \bu_t)(\cdot,t)\|^2_{\bL^2}
+\|(\sqrt{\rho} \bw_t)(\cdot,t)\|^2_{\bL^2}\Big]
\nonumber
\\&&
\qquad\qquad
+
\Uppsi_2
\Big(
1+\max\Big\{(\|h(\cdot,t)\|_{\bH^2},1\Big\}
\|\nabla\rho(\cdot,t)\|^2_{\bL^q}
\Big)
\Big[\|f(t)\|^2+\|g(t)\|^2\Big]
\nonumber\\&&
\qquad\qquad
+
\Uppsi_3
\Big[\|f'(t)\|^2+\|g'(t)\|^2\Big]
+\Uppsi_4
\Big[
\|\bu(\cdot,t)\|^{6/q}_{\bH^2}
+\|\bw(\cdot,t)\|^{6/q}_{\bH^2}
\Big]\|\nabla\rho(\cdot,t)\|^2_{L^q}
\label{eq:lem:estimates3}
\end{eqnarray}
for all $t\in [0,T_1]$ with $T_1$ as is given on Lemma~\ref{lem:u_ut_w_wt_estimate}.
\end{lemma}

\begin{proof}
Differentiating \eqref{eq:momento_lineal} 
and  \eqref{eq:momento_angular} with respect to $t$; 
testing the results by $\bu_t$ and  $\bw_t$,
respectively;
summing the resulting equations; and rearranging the terms we get
\begin{eqnarray}
&&\frac{1}{2}\frac{d}{dt}\int_{\Omega}\rho|\bu_t(\bx,t)|^2d\bx
+(\mu+\mu_r)\int_{\Omega}|\nabla \bu_t(\bx,t)|^2d\bx
+\frac{1}{2}\frac{d}{dt}\int_{\Omega}\rho| \bw_t(\bx,t)|^2d\bx
\nonumber\\&&
\quad
+(c_a+2c_d)\int_{\Omega}|\nabla \bw_t(\bx,t)|^2d\bx
+\int_{\Omega}\Big(\frac{1}{2}\rho(\bx,t)+4\mu_r\Big)|w_t(\bx,t)|^2d\bx
\nonumber\\&&
\quad
=2\mu_r\left[\int_{\Omega}\curl \bw_t(x,t) \cdot\bu_t(\bx,t)d\bx
+\int_{\Omega}\curl \bu_t(x,t)\cdot\bw_t(\bx,t)d\bx\right]
\nonumber\\&&
\quad\quad\quad
+\left[\int_{\Omega}f'(t)\Big(\rho(\nabla h-\bm)\cdot\bu_t\Big)(\bx,t)d\bx
+\int_{\Omega}g'(t)\Big(\rho \bq\cdot\bw_t\Big)(\bx,t)d\bx\right]
\nonumber\\&&
\quad\quad\quad
+\left[\int_{\Omega}f(t)\Big(\rho(\nabla h_t-\bm_t)\cdot\bu_t\Big)(\bx,t)d\bx
+\int_{\Omega}g(t)\Big(\rho\bq_t\cdot\bw_t\Big)(\bx,t)d\bx\right]
\nonumber\\&&
\quad\quad\quad
+\left[
\int_{\Omega}f(t)\Big(\rho_t(\nabla h-\bm)\cdot\bu_t\Big)(\bx,t)d\bx
+\int_{\Omega}g(t)\Big(\rho_t\bq \cdot\bw_t\Big)(\bx,t)d\bx
\right]
\nonumber\\&&
\quad\quad\quad
-\frac{1}{2}\left[\int_{\Omega}\Big(\rho_t |\bu_t|^2\Big)(\bx,t)d\bx
+\int_{\Omega}\Big(\rho_t |\bw_t|^2\Big)(\bx,t)d\bx\right]
\nonumber\\&&
\quad\quad\quad
-\left[
\int_{\Omega}\Big(\rho_t (\bu\cdot\nabla )\bu\cdot \bu_t\Big)(\bx,t)d\bx
+\int_{\Omega}\Big(\rho_t (\bu\cdot\nabla )\bw \cdot\bw_t\Big)(\bx,t)d\bx
\right]
\nonumber\\&&
\quad\quad\quad
-\left[
\int_{\Omega}\Big(\rho (\bu_t\cdot\nabla )\bu\cdot \bu_t\Big)(\bx,t)d\bx
+\int_{\Omega}\Big(\rho (\bu_t\cdot\nabla )\bw\cdot \bw_t\Big)(\bx,t)d\bx
\right]
\nonumber\\&&
\quad
= \sum_{i=0}^6 I_i,
\label{eq:lem:estimates3_full}
\end{eqnarray}
where $I_i$ for 
$i=0,\ldots,6$ are defined by the brackets $[\ldots]$.
Hence, the proof of \eqref{eq:lem:estimates3} is reduced to
get some bounds for 
each $I_i$ as will be specified below.

\vskip 2mm
\noindent
{\it Estimate for $I_0$.}
From Lemmas~\ref{lema:aprioriestimatesforrho}, \ref{lem:u_ut_w_wt_estimate}
and Young inequality, we find that $I_0$ can be bounded as follows
\begin{eqnarray}
I_0
&\le& 2\mu_r\left|\int_{\Omega}\Big(\curl \bw_t\cdot\bu_t\Big)(\bx,t)d\bx
+\int_{\Omega}\Big(\curl \bu_t\cdot\bw_t\Big)(\bx,t)d\bx\right|
\nonumber
\\
\quad
&\le& 2\mu_r\Big(\|\nabla \bw_t(\cdot,t)\|_{\bL^2}\|\bu_t(\cdot,t)\|_{\bL^2}
+\|\nabla \bu_t(\cdot,t)\|_{\bL^2}\|\bw_t(\cdot,t)\|_{\bL^2}\Big)
\nonumber
\\
\quad
&\le &
\hat{\epsilon}
\Big[\|\nabla \bu_t(\cdot,t)\|^2_{\bL^2}
+\|\nabla \bw_t(\cdot,t)\|^2_{\bL^2}\Big]
+\frac{\Upphi_0}{4\hat{\epsilon}}
\Big[\|(\sqrt{\rho} \bu_t)(\cdot,t)\|^2_{\bL^2}
+\|(\sqrt{\rho} \bw_t)(\cdot,t)\|^2_{\bL^2}\Big]
,
\label{eq:lem:estimates3_I0}
\end{eqnarray}
for all $t\in [0,T_1]$, where 
$\Upphi_0=(2\mu_r\alpha^{-1})^2$ and
$\hat{\epsilon}>0$.

\vskip 2mm
\noindent
{\it Estimate for $I_1$.}
By Lemmas~\ref{lema:aprioriestimatesforrho},
\ref{lema:aprioriestimatesforhandr} and \eqref{eq:yung}, we get that
\begin{eqnarray}
I_1
&\le&\left|\int_{\Omega}f'(t)\Big(\rho(\nabla h-\bm)\cdot\bu_t\Big)(\bx,t)d\bx
    +\int_{\Omega}g'(t)\Big(\rho \bq\cdot\bw_t\Big)(\bx,t)d\bx\right|
\nonumber
\\
\quad
&\le& \sqrt{\beta}\Big(
|f'(t)|\|(\nabla h-\bm)(\cdot,t)\|_{\bL^2}
\|\sqrt{\rho} \bu_t(\cdot,t)\|_{\bL^2}
+|g'(t)|\|\bq(\cdot,t)\|_{\bL^2}\|\sqrt{\rho} \bw_t(\cdot,t)\|_{\bL^2}\Big)
\nonumber
\\
\quad
&\le &
\Upphi_1
\Big(|f'(t)|^2+|g'(t)|^2
+
\|(\sqrt{\rho} \bu_t)(\cdot,t)\|^2_{\bL^2}
+
\|(\sqrt{\rho} \bw_t)(\cdot,t)\|^2_{\bL^2}\Big),
\label{eq:lem:estimates3_I1}
\end{eqnarray}
where
$
\Upphi_1=2^{-1}\sqrt{\beta}
\max\Big\{\Uppi_1+\|\bm(\cdot,t)\|_{\bL^2},\|\bq(\cdot,t)\|_{\bL^2}\Big\}.
$

\vskip 2mm
\noindent
{\it Estimate for $I_2$.}
By applying Lemma~\ref{lema:aprioriestimatesforrho} and \eqref{eq:yung}, 
we have that
\begin{eqnarray}
I_2
&\le &\left|\int_{\Omega}\Big(f\rho(\nabla h_t-\bm_t)\cdot\bu_t\Big)(\bx,t)d\bx
    +\int_{\Omega}\Big(g\rho \bq_t\cdot\bw_t\Big)(\bx,t)d\bx\right|
\nonumber
\\
\quad
&\le& \sqrt{\beta}\Big(
|f(t)|
\Big\{\|\nabla h_t(\cdot,t)\|_{\bL^2}+\|\bm_t(\cdot,t)\|_{\bL^2}\Big\}
\|\sqrt{\rho} \bu_t(\cdot,t)\|_{\bL^2}
+|g(t)|\|\bq_t(\cdot,t)\|_{\bL^2}\|\sqrt{\rho} \bw_t(\cdot,t)\|_{\bL^2}\Big),
\nonumber
\\
&\le&  
\Upphi_2
\Big(
|f(t)|^2+|g(t)|^2
+\max\Big\{\Big(\|\nabla h_t(\cdot,t)\|_{\bL^2}+1\Big)^2,1\Big\}
\Big[
\|\sqrt{\rho} \bu_t(\cdot,t)\|^2_{\bL^2}
+
\|\sqrt{\rho} \bw_t(\cdot,t)\|^2_{\bL^2}\Big]
\Big).
\label{eq:lem:estimates3_I2}
\end{eqnarray}
where
$
\Upphi_2=\sqrt{\beta}\;
\max\Big\{\|\bm_t(\cdot,t)\|_{\bL^2},\;
\|\bq_t(\cdot,t)\|_{\bL^2},\;1\Big\}.
$

\vskip 2mm
\noindent
{\it Estimate for $I_3$.}
By  equation \eqref{eq:ecuacion_continuidad},
inequalities \eqref{eq:poi} and \eqref{eq:gagnir},
Lemmas~\ref{lema:aprioriestimatesforrho} and
\ref{lem:u_ut_w_wt_estimate} and noticing  that
\begin{eqnarray*}
\|\nabla \rho(\cdot,t)\|_{\bL^3}& \le & |\Omega|^{(q-3)/3q}
\|\nabla \rho(\cdot,t)\|_{\bL^q}, \qquad q>3,
\\
\|(\nabla h-\bm)(\cdot,t)\|^2_{\bL^3}&\le&
C^2_{gn}
\|(\nabla h-\bm)(\cdot,t)\|_{\bL^2}\|(\nabla h-\bm)(\cdot,t)\|_{\bH^1}
\\
&\le& 
\Big(\Uppi_1+\|\bm(\cdot,t)\|_{\bL^2}\Big)
\Big(\|h(\cdot,t)\|_{H^2}
+\|\bm(\cdot,t)\|_{\bH^1}\Big)
\end{eqnarray*}
we deduce that
\begin{eqnarray}
I_3
&\le&\left|\int_{\Omega}f(t)\Big(\rho_t(\nabla h-\bm)\cdot\bu_t\Big)(\bx,t)d\bx
    +\int_{\Omega}g(t)\Big(\rho_t\bq\cdot\bw_t\Big)(\bx,t)d\bx\right|
\nonumber
\\
\quad
&=&\left|\int_{\Omega}f(t)\Big((\bu\cdot\nabla\rho)(\nabla h-\bm)\cdot\bu_t\Big)(\bx,t)d\bx
    +\int_{\Omega}g(t)\Big((\bu\cdot\nabla\rho)\bq\cdot\bw_t\Big)(\bx,t)d\bx\right|
\nonumber
\\
\quad
&\le& 
|f(t)|
\|\bu(\cdot,t)\|_{\bL^6}\|\nabla\rho(\cdot,t)\|_{\bL^3}
\|(\nabla h-\bm)(\cdot,t)\|_{\bL^3}\|\bu_t(\cdot,t)\|_{\bL^6}
\nonumber
\\&&
\qquad
+|g(t)|\|\bu(\cdot,t)\|_{\bL^6}\|\nabla\rho(\cdot,t)\|_{\bL^3}
\|\bq(\cdot,t)\|_{\bL^3}\|\bw_t(\cdot,t)\|_{\bL^6}
\nonumber
\\
\quad
&\le&
(C_{poi})^2
\Uptheta(t)|\Omega|^{(q-3)/3q}
\Big\{
|f(t)|\|\nabla\rho(\cdot,t)\|_{\bL^q}
\|(\nabla h-\bm)(\cdot,t)\|_{\bL^3}\|\nabla\bu_t(\cdot,t)\|_{\bL^2}
\nonumber
\\&&
\qquad
+|g(t)|\|\nabla\rho(\cdot,t)\|_{\bL^3}
\|\bq(\cdot,t)\|_{\bL^3}\|\nabla\bw_t(\cdot,t)\|_{\bL^2}\Big\}
\nonumber
\\
\quad
&\le&
\hat{\epsilon}\Big[\|\nabla \bu_t(\cdot,t)\|^2_{\bL^2}+\|\nabla \bw_t(\cdot,t)\|^2_{\bL^2} \Big]
+
\frac{\Big((C_{poi})^2\Uptheta(t)|\Omega|^{(q-3)/3q} \Big)^2}{4\hat{\epsilon}}
\times
\nonumber
\\
&&
\qquad\qquad
\Big\{
|f(t)|^2
\|\nabla\rho(\cdot,t)\|^2_{\bL^p}
\|(\nabla h-\bm)(\cdot,t)\|^2_{\bL^2}
+
|g(t)|^2
\|\nabla\rho(\cdot,t)\|^2_{\bL^p}
\|\bq(\cdot,t)\|^2_{\bL^2}
\Big\}
\nonumber
\\
\quad
&\le&
\hat{\epsilon}\Big[\|\nabla \bu_t(\cdot,t)\|^2_{\bL^2}+\|\nabla \bw_t(\cdot,t)\|^2_{\bL^2} \Big]
+
\Upphi_3
\max\Big\{ \|h\|_{H^2}+1,1\Big\}
\|\nabla\rho(\cdot,t)\|^2_{\bL^p}
\Big [|f(t)|^2+|g(t)|^2\Big ],
\label{eq:lem:estimates3_I3}
\end{eqnarray}
where $\hat{\epsilon}>0$ and
\begin{eqnarray*}
 \Upphi_3=
\Big((C_{poi})^2\Uptheta(t)|\Omega|^{(q-3)/3q} \Big)^2
\max\Big\{\Uppi_1+\|\bm(\cdot,t)\|_{\bL^2},
(\Uppi_1+\|\bm(\cdot,t)\|_{\bL^2})\|\bm(\cdot,t)\|_{\bL^2},
\|\bq(\cdot,t)\|_{\bL^2}\|\bq(\cdot,t)\|_{\bH^1}\Big\},
\end{eqnarray*}
for all $t\in[0,T_1]$.

\vskip 2mm
\noindent
{\it Estimate for $I_4$.}
It  can be bounded by the application of
equation \eqref{eq:ecuacion_continuidad}, \eqref{eq:poi}, \eqref{eq:gagnir}
and Lemma~\ref{lem:u_ut_w_wt_estimate}, since
we can perform the following calculus
\begin{eqnarray}
I_4
&\le& \frac{1}{2}\left|\int_{\Omega}\Big(\rho_t |\bu_t|^2\Big)(\bx,t)d\bx
    +\int_{\Omega}\Big(\rho_t |\bw_t|^2\Big)(\bx,t)d\bx\right|
\nonumber
\\
\quad
&=&\frac{1}{2}\left|\int_{\Omega}\Big((\bu\cdot\nabla\rho) |\bu_t|^2\Big)(\bx,t)d\bx
    +\int_{\Omega}\Big((\bu\cdot\nabla\rho) |\bw_t|^2\Big)(\bx,t)d\bx\right|
\nonumber
\\
\quad
&\le&
\|\bu(\cdot,t)\|_{\bL^6}\|\nabla\rho(\cdot,t)\|_{\bL^q}
\|\bu_t(\cdot,t)\|_{\bL^{2q/(q-2)}}
\|\bu_t(\cdot,t)\|_{\bL^3}
\nonumber
\\&&
\qquad
+\|\bu(\cdot,t)\|_{\bL^6}\|\nabla\rho(\cdot,t)\|_{\bL^q}
\|\bw_t(\cdot,t)\|_{\bL^{2q/(q-2)}}
\|\bw_t(\cdot,t)\|_{\bL^3}
\nonumber
\\
\quad
&\le&
\frac{(C_{poi})^2\Uptheta(t)}{2}\Big\{
\|\nabla\rho(\cdot,t)\|_{\bL^q}\|\bu_t(\cdot,t)\|_{\bL^{2q/(q-2)}}
\|\bu_t(\cdot,t)\|_{\bL^3}
\nonumber
\\&&
\qquad
+\|\nabla\rho(\cdot,t)\|_{\bL^q}\|\bw_t(\cdot,t)\|_{\bL^{2q/(q-2)}}
\|\bw_t(\cdot,t)\|_{\bL^3}
\Big\}
\nonumber
\\
\quad
&\le&
\frac{(C_{poi}C_{gn})^2\Uptheta(t)}{2}
 \Big\{
\|\nabla\rho(\cdot,t)\|_{\bL^q}
\|\bu_t(\cdot,t)\|^{1-3/q}_{\bL^2}\|\nabla \bu_t(\cdot,t)\|^{3/q}_{\bL^2}
\|\bu_t(\cdot,t)\|^{1/2}_{\bL^2}\|\nabla \bu_t(\cdot,t)\|^{1/2}_{\bL^2}
\nonumber
\\&&
\qquad
+\|\nabla\rho(\cdot,t)\|_{\bL^q}
\|\bw_t(\cdot,t)\|^{1-3/q}_{\bL^2}\|\nabla \bw_t(\cdot,t)\|^{3/q}_{\bL^2}
\|\bw_t(\cdot,t)\|^{1/2}_{\bL^2}\|\nabla \bw_t(\cdot,t)\|^{1/2}_{\bL^2}
\Big\}
\nonumber
\\
&\le&
\hat{\epsilon}\Big[
\|\nabla \bu_t(\cdot,t)\|^{2}_{\bL^2}+\|\nabla \bw_t(\cdot,t)\|^{2}_{\bL^2}
\Big]
+
\Upphi_4
\Big[
\|(\sqrt{\rho}\bu_t)(\cdot,t)\|^{2}_{\bL^2}
+\|(\sqrt{\rho}\bw_t)(\cdot,t)\|^{2}_{\bL^2}
\Big]\|\nabla\rho(\cdot,t)\|^{4q/(3q-6)}_{\bL^q}
\label{eq:lem:estimates3_I4}
\end{eqnarray}
where $q>3$,
$\hat{\epsilon}>0$ and
$\Upphi_4=
\left(2^{-1}(C_{poi}C_{gn})^2\Uptheta(t)\right)^{4q/(3q-6)} \hat{C}_{\hat{\epsilon}}
$
with $\hat{C}_{\hat{\epsilon}}$ defined in \eqref{eq:poi} for 
the conjugate values $4q/(6+q)$ and $4q/(3q-6)$ instead of $p$
and $q$, respectively.

\vskip 2mm
\noindent
{\it Estimate for $I_5$.}
An application of
equation \eqref{eq:ecuacion_continuidad}, 
\eqref{eq:poi}, \eqref{eq:gagnir} 
and Lemma~\ref{lem:u_ut_w_wt_estimate} implies 
the following bound for~$I_5$
\begin{eqnarray}
I_5
&\le& \left|\int_{\Omega}\Big(\rho_t (\bu\cdot\nabla)\bu\cdot \bu_t\Big)(\bx,t)d\bx
    +\int_{\Omega}\Big(\rho_t\bu\cdot\nabla)\bw \cdot\bw_t\Big)(\bx,t)d\bx\right|
\nonumber
\\
\quad
&=&\left|\int_{\Omega}\Big((\bu\cdot\nabla\rho)  (\bu\cdot\nabla)\bu\cdot\bu_t\Big)(\bx,t)d\bx
    +\int_{\Omega}\Big((\bu\cdot\nabla\rho) (\bu\cdot\nabla)\bw\cdot\bw_t\Big)(\bx,t)d\bx\right|
\nonumber
\\
\quad
&\le&
\|\bu(\cdot,t)\|^2_{\bL^6}\|\bu_t(\cdot,t)\|_{\bL^6}\|\nabla\rho(\cdot,t)\|_{\bL^q}
\|\nabla \bu(\cdot,t)\|_{\bL^{2q/(q-2)}}
\nonumber
\\&&
\qquad
+\|\bu(\cdot,t)\|^2_{L^6}\|\bw_t(\cdot,t)\|_{\bL^6}\|\nabla\rho(\cdot,t)\|_{\bL^q}
\|\nabla \bw(\cdot,t)\|_{\bL^{2q/(q-2)}}
\nonumber
\\
\quad
&\le&
C^3_{poi}\Uptheta(t)^2 C_{gn}\Big\{
\|\nabla \bu_t(\cdot,t)\|_{\bL^2}\|\nabla\rho(\cdot,t)\|_{\bL^q}
\|\nabla \bu(\cdot,t)\|^{1-3/q}_{\bL^2}\|\bu(\cdot,t)\|^{3/q}_{\bH^2}
\nonumber
\\&&
\qquad
+\|\nabla \bw_t(\cdot,t)\|_{\bL^2}\|\nabla\rho(\cdot,t)\|_{\bL^q}
\|\nabla \bw(\cdot,t)\|^{1-3/q}_{\bL^2}\|\bw(\cdot,t)\|^{3/q}_{\bH^2}
\Big\}
\nonumber\\
&\le&
C^3_{poi}[\Uptheta(t)]^{3(q-1)/q} C_{gn}
\Big\{
\|\nabla \bu_t(\cdot,t)\|_{\bL^2}\|\nabla\rho(\cdot,t)\|_{\bL^q}
\|\bu(\cdot,t)\|^{3/q}_{\bH^2}
\nonumber\\
&&
\qquad
+
\|\nabla \bw_t(\cdot,t)\|_{\bL^2}\|\nabla\rho(\cdot,t)\|_{\bL^q}
\|\bw(\cdot,t)\|^{3/q}_{\bH^2}\Big\}
\nonumber\\
\quad&\le&
\hat{\epsilon}
\Big[
\|\nabla \bu_t(\cdot,t)\|^{2}_{\bL^2}+\|\nabla \bw_t(\cdot,t)\|^{2}_{\bL^2}
\Big]
+\frac{\Upphi_5}{4\hat{\epsilon}}
\Big[
\|\bu(\cdot,t)\|^{6/q}_{\bH^2}+\|\bw(\cdot,t)\|^{6/q}_{\bH^2}
\Big]
\|\nabla\rho(\cdot,t)\|^2_{\bL^q}
,
\label{eq:lem:estimates3_I5}
\end{eqnarray}
where $\Upphi_5=\Big(C^3_{poi}[\Uptheta(t)]^{3(q-1)/q} C_{gn}\Big)^2$
with $\hat{\epsilon}>0$.

\vskip 2mm
\noindent
{\it Estimate for $I_6$.}
By inequalities \eqref{eq:poi}, \eqref{eq:gagnir},
Lemma~\ref{lema:aprioriestimatesforrho} and
Lemma~\ref{lem:u_ut_w_wt_estimate} we deduce that
\begin{eqnarray}
I_6
&\le&\left|\int_{\Omega}\Big(\rho (\bu_t\cdot\nabla)\bu\cdot \bu_t\Big)(\bx,t)d\bx
    +\int_{\Omega}\Big(\rho (\bu_t\cdot\nabla)\bw \cdot\bw_t\Big)dx\right|
\nonumber
\\
\quad
&\le&
\|\rho(\cdot,t)\|_{L^\infty}\|\nabla \bu(\cdot,t)\|_{\bL^2}
\| \bu_t(\cdot,t)\|_{L^3}\| \bu_t(\cdot,t)\|_{\bL^6}
+\|\rho(\cdot,t)\|_{L^\infty}\|\nabla \bu(\cdot,t)\|_{\bL^2}
\| \bw_t(\cdot,t)\|_{\bL^3}\| \bw_t(\cdot,t)\|_{\bL^6}
\nonumber
\\
\quad
&\le&
\beta C_{poi}C_{gn}\Uptheta(t) \Big\{
\| \bu_t(\cdot,t)\|^{1/2}_{\bL^2}\| \nabla \bu_t(\cdot,t)\|^{3/2}_{\bL^2}
+\| \bu_t(\cdot,t)\|^{1/2}_{\bL^2}
\| \nabla\bu_t(\cdot,t)\|^{1/2}_{\bL^2}
\| \nabla \bw_t(\cdot,t)\|_{\bL^2}
\Big\}
\nonumber
\\
\quad&\le&
2\hat{\epsilon}
\Big[\|\nabla \bu_t(\cdot,t)\|^{2}_{\bL^2}
+
\|\nabla \bw_t(\cdot,t)\|^{2}_{\bL^2}\Big]
+
\frac{1}{4\alpha}
\left(
\frac{3\sqrt{2}(\Upphi_6)^{4/3}}{\sqrt[3]{\hat{\epsilon}}}
+
\frac{(\Upphi_6)^{4}}{16\hat{\epsilon}^2}
\right)
\|(\sqrt{\rho}\bu_t)(\cdot,t)\|^{2}_{\bL^2},
\qquad
\label{eq:lem:estimates3_I6}
\end{eqnarray}
where $\Upphi_6=\beta C_{poi}C_{gn}\Uptheta(t) $
with $\hat{\epsilon}>0$.

Inserting \eqref{eq:lem:estimates3_I1}-\eqref{eq:lem:estimates3_I6} in
\eqref{eq:lem:estimates3_full}, selecting 
$\hat{\epsilon}\in ]0, 6^{-1}\min\{\mu+\mu_r,c_a+2c_d\}[$
and defining
\begin{eqnarray*}
&&\Uppsi_1 =\frac{1}{\mathcal{L}}
\max\left\{\frac{\Upphi_0}{4\hat{\epsilon}}+\Upphi_1
+\frac{1}{4\alpha}
\left(
\frac{3\sqrt{2}(\Upphi_6)^{4/3}}{\sqrt[3]{\hat{\epsilon}}}
+
\frac{(\Upphi_6)^{4}}{16\hat{\epsilon}^2}
\right),
\;
\Upphi_2,
\;
\Upphi_4\right\}
\\
&&
\Uppsi_2 =\frac{1}{\mathcal{L}}\max\Big\{\Upphi_2,\;\Upphi_3\},
\qquad
\Uppsi_3 =\frac{\Upphi_1}{\mathcal{L}},
\qquad
\Uppsi_4 =\frac{\Upphi_1}{4\hat{\epsilon}\mathcal{L}},
\end{eqnarray*}
with $\mathcal{L}=\min\{2^{-1},\mu+\mu_r-\hat{\epsilon},c_a+2c_d-\hat{\epsilon}\}$,
we can deduce that \eqref{eq:lem:estimates3} is satisfied.
\end{proof}

\subsection{Proof of Theorem~\ref{teo:global_estimates}}
\label{subsec:proof_of_teo:global_estimates}

We note that \eqref{eq:lema:aprioriestimatesforrho} is clearly valid by 
Lemma~\ref{lema:aprioriestimatesforrho} and 
the existence of $T_1$ and $\upkappa_1$ follows from 
\eqref{lem:u_ut_w_wt_main_estimate}.  Now, before starting the
proof of \eqref{teo:global_estimates_kapa2}-\eqref{teo:global_estimates_kapa14},
we deduce two estimates. First, differentiating \eqref{eq:ecuacion_continuidad}
with respect to $x_i$, using \eqref{eq:incompresibilidad}, testing the result 
by $|\rho_{x_i}|^{q-2}\rho_{x_i}$ and applying the Sobolev inequality we 
deduce that there exists $C_{sob}$ independent of $f$ and $g$ such that
\begin{eqnarray}
\frac{d}{dt}\|\nabla\rho(\cdot,t)\|^q_{\bL^q}\le C_{sob} 
\|\bu(\cdot,t)\|_{\bW^{2,s}}\|\nabla\rho(\cdot,t)\|_{\bL^q},\quad
\mbox{for $t\in [0,T_*]$.}
\label{eq:sobolev}
\end{eqnarray}
Second, by the regularity of the solutions for \eqref{eq:stokes_problem}
we have that there exists $C^{reg}_3$ depending only on $\mu,\mu_r$ and
$\Omega$ such that
\begin{eqnarray*}
\|\bu(\cdot,t)\|_{\bW^{2,s}}+\|p(\cdot,t)\|_{\bW^{1,s}}
\le C^{reg}_3\; \Big\|\Big(2\mu_r \curl \bw
+\rho f(\nabla h-\bm)
-\rho \bu_t-\rho (\bu\cdot\nabla) \bu\Big)(\cdot,t)\Big\|_{\bL^s},
\end{eqnarray*} 
for $s\in [2,\infty[$. 
We note that for $s\in [2,6]$ the inequalities
 \eqref{eq:poi} and  \eqref{eq:gagnir} can be applied.
Hence, by the Minkowski and H\"older inequalities, \eqref{eq:poi} and  \eqref{eq:gagnir},
we find that there exists 
$\upxi_1=2\mu_rC^{reg}_3C_{gn},$
$\upxi_2=\beta C^{reg}_3\max\{C_{gn},\|\bm(\cdot,t)\|_{\bL^2}\}$, 
$\upxi_3=C^{reg}_3C_{gn}C^{2,\infty}_{iny}$ and 
$\upxi_4=\beta C^{reg}_3C_{poi}$, such that
\begin{eqnarray}
\|\bu(\cdot,t)\|_{\bW^{2,s}}+\|p(\cdot,t)\|_{W^{1,s}}
&\le 
\upxi_1\|\bw(\cdot,t)\|_{\bH^2}
+\upxi_2|f(t)|\Big(\|h(\cdot,t)\|_{H^2}+1\Big)
\nonumber\\
&
+\upxi_3\|\bu(\cdot,t)\|^2_{\bH^2}
+\upxi_4 \|\nabla \bu_t(\cdot,t)\|_{\bL^2},
\quad
s\in [2,6],
\label{eq:w2s_regularity}
\end{eqnarray}
for $t\in [0,T_*]$.
Therefore, we derive the proof of  \eqref{teo:global_estimates_kapa2} by
inserting \eqref{eq:w2s_regularity} in \eqref{eq:sobolev} and using the estimates
\eqref{h_and_r_estimate}, \eqref{ht_and_rt_estimate} and \eqref{eq:lem:estimates3}. 
The estimate \eqref{teo:global_estimates_kapa3_4} is deduced
from \eqref{teo:global_estimates_kapa2} and
\eqref{h_and_r_estimate}.
The inequality \eqref{teo:global_estimates_kapa5_8} is obtained 
from \eqref{teo:global_estimates_kapa2}, \eqref{eq:uppi_cuatro}, \eqref{eq:uppi_cinco},
\eqref{ht_and_rt_estimate} and \eqref{eq:regularity_u_p_w}.
The estimate \eqref{teo:global_estimates_kapa9_10} is proved by the application of 
\eqref{teo:global_estimates_kapa1}, \eqref{eq:uppi_cuatro}, \eqref{eq:uppi_cinco} and
\eqref{eq:regularity_u_p_w}.
The estimate \eqref{teo:global_estimates_kapa11_13} follows
from \eqref{eq:ecuacion_continuidad}, \eqref{teo:global_estimates_kapa1}
and \eqref{teo:global_estimates_kapa9_10}.
We complete the proof of the theorem deducing the inequality
\eqref{teo:global_estimates_kapa14} 
by combining the results given on
\eqref{teo:global_estimates_kapa2} and \eqref{eq:w2s_regularity}.

\section*{Acknowledgment}
We acknowledge the support of the research projects 
DIUBB GI 172409/C, DIUBB 183309 4/R and FAPEI at
Universidad del B{\'\i}o-B{\'\i}o (Chile); the 
Fondecyt project 1120260; the project MTM 2012-32325 (Spain);
and CONICYT (Chile) through the program ``Becas de Doctorado''




\end{document}